\documentclass[a4paper, 10pt, notitlepage]{article}

\usepackage{amsthm, amsmath, amssymb, latexsym, eufrak}
\usepackage{mathrsfs,cite}
\usepackage[multiple]{footmisc}
\usepackage{graphicx}
\usepackage[ruled,vlined]{algorithm2e}

\theoremstyle{plain}
\newtheorem{theorem}{Theorem}
\newtheorem{lemma}{Lemma}
\newtheorem{corollary}{Corollary}
\newtheorem{proposition}{Proposition}

\theoremstyle{definition}

\newtheorem{example}{Example}

\theoremstyle{remark}
\newtheorem{remark}{Remark}

\DeclareMathOperator{\co}{co}
\DeclareMathOperator{\determ}{det}
\DeclareMathOperator{\trace}{Tr}
\DeclareMathOperator{\sign}{sign}
\DeclareMathOperator{\diag}{diag}

\author{M.V. Dolgopolik\footnote{Institute for Problems in Mechanical Engineering of the Russian Academy of Sciences, 
Saint Petersburg, Russia}}
\title{Subdifferentials of convex matrix-valued functions}

\begin{document}

\maketitle

\begin{abstract}
Subdifferentials (in the sense of convex analysis) of matrix-valued functions defined on $\mathbb{R}^d$ that are convex
with respect to the L\"{o}wner partial order can have a complicated structure and might be very difficult to compute
even in simple cases. The aim of this paper is to study subdifferential calculus for such functions and properties of
their subdifferentials. We show that many standard results from convex analysis no longer hold true in the matrix-valued
case. For example, in this case the subdifferential of the sum is not equal to the sum of subdifferentials, the Clarke
subdifferential is not equal to the subdifferential in the sense of convex analysis, etc. Nonetheless, it is possible to
provide simple rules for computing nonempty subsets of subdifferentials (in particular, individual subgradients) of
convex matrix-valued functions in the general case and to completely describe subdifferentials of such functions defined
on the real line. As a by-product of our analysis, we derive some interesting properties of convex matrix-valued
functions, e.g. we show that if such function is nonsmooth, then its diagonal elements must be nonsmooth as well.
\end{abstract}

\section{Introduction}

Matrix-valued functions that are convex with respect to the L\"{o}wner partial order naturally arise in the context of
semidefinite programming and semidefinite complementarity problems
\cite{Dolgopolik_DC_Semidef_I,Dolgopolik_DC_Semidef_II,ChenQiTseng,SunSun} and play an important role in the design of
certain classes of numerical methods for solving these problems. Although basic properties of such mappings have been
thoroughly investigated (see \cite{HansenTomiyama,BrinkhuisLuoZhang} and the references therein), relatively little
attention has been paid to analysis of differentiability properties of such mappings in the general nonsmooth case. 

A detailed analysis of nonsmooth matrix-valued functions mapping the space of symmetric matrices into itself was
presented in \cite{ChenQiTseng}, while semismoothness of certain important matrix-valued functions arising in
applications have been proved in \cite{SunSun}. However, to the best of the author's knowledge, subdifferentials of
nonsmooth convex matrix-valued functions have not been properly studied before.

It should be noted that general theory of convex mappings with values in normed/vector spaces, developed, e.g. in 
\cite{Papageorgiou,Thera,KusraevKutateladze}, cannot be applied to convex matrix-valued functions, since the space of
symmetric matrices endowed with the L\"{o}wner partial order is \textit{not} a vector lattice by the famous Kadison's
theorem \cite{Kadison} and is not normal (see \cite{Papageorgiou,Thera}). The lack of these properties does not allow
one to extend and apply many well-developed techniques of convex analysis to the study of convex matrix-valued
functions. Therefore, different techniques for analysing such functions must be developed.

The main goal of this article is to present a partial extension of the subdifferential calculus from convex analysis to
the matrix-valued case and provide some simple rules for computing nonempty subsets (in particular, individual
subgradients) of nonsmooth convex matrix-valued functions.

We show that many standard results from convex analysis no longer hold true in the matrix-valued case, e.g. 
the subdifferential of the sum is not equal to the sum of the subdifferentials, the Clarke subdifferential is not equal
to the subdifferential in the sense of convex analysis, etc. Nonetheless, we prove that subdifferentials of convex
matrix-valued functions possess many properties of their real-valued counterparts. We also provide a complete
characterisation of subdifferentials of convex matrix-valued functions defined on the real line and present a partial
extension of the subdiferential calculus to the matrix-valued case that allows one to compute nonempty subsets of
subdifferentials of many nonsmooth convex matrix-valued functions in a straightforward manner. 

Finally, as a by-product of our analysis we derive some differentiability properties of nonsmooth convex matrix-valued
functions that have no direct analogues in the real-valued case. In particular, we show that diagonal matrix-valued
functions can have subgradients that consist of non-diagonal matrices and prove that nonsmooth convex matrix-valued
functions must have nonsmooth diagonal elements.

The paper is organised as follows. Some auxiliary definitions and results related to the L\"{o}wner partial order and
convex matrix-valued functions are collected in Section~\ref{sect:Prelim}. General properties of subdifferentials of
such functions are studied in Section~\ref{sect:GeneralProperties}. A characterisation of subdifferentials of convex
matrix-valued functions defined on the real line is presented in Section~\ref{sect:OneDimCase}, while a relationship
between the Clarke subdifferential and the subdifferential in the sense of convex analysis is studied in
Section~\ref{sect:MultiDimCase}. Finally, Section~\ref{sect:SubdiffCalc} is devoted to a partial extension of the
subdifferential calculus from convex analysis to the matrix-valued case.

\section{Preliminaries}
\label{sect:Prelim}

Let us recall some auxiliary definitions and results that will be utilised throughout the article. The space of all
real symmetric matrices of order $\ell \in \mathbb{N}$ is denoted by $\mathbb{S}^{\ell}$. We endow this space with 
the standard inner product $\langle A, B \rangle = \trace(AB)$ for all $A, B \in \mathbb{S}^{\ell}$ and 
the corresponding norm $\| \cdot \|_F$, which is called the Frobenius norm. Here $\trace(A)$ is the trace of a square
matrix $A$.

Denote by $\preceq$ the L\"{o}wner partial order on the space $\mathbb{S}^{\ell}$, which is defined as follows:
$A \preceq B$ for some matrices $A, B \in \mathbb{S}^{\ell}$ if and only if the matrix $B - A$ is positive semidefinite.
It is worth explicitly noting that by the definition of the L\"{o}wner partial order for any matrices 
$A, B \in \mathbb{S}^{\ell}$ the following two conditions are equivalent:
\begin{equation} \label{eq:LownerOrderViaQuadForm}
  A \preceq B \quad \Longleftrightarrow \quad 
  \langle z, (B - A) z \rangle \ge 0 \quad \forall z \in \mathbb{R}^{\ell},
\end{equation}
where $\langle \cdot, \cdot \rangle$ is the inner product in $\mathbb{R}^{\ell}$. 

It is readily seen that the L\"{o}wner partial order enjoys the following properies:
\begin{equation} \label{eq:LoewnerOrderProp}
\begin{split}
  &A_1 \preceq B_1 \Longleftrightarrow - B_1 \preceq - A_1, \quad
  A_1 \preceq B_1 \implies A_1 + E \preceq B_1 + E
  \\
  &A_1 \preceq B_1 \quad \& \quad A_2 \preceq B_2 \implies A_1 + A_2 \preceq B_1 + B_2,
  \\
  &A_1 \preceq B_1 \implies t A_1 \preceq t B_1 \quad \forall t \ge 0.
\end{split}
\end{equation}
for any matrices $A_i, B_i, E \in \mathbb{S}^{\ell}$. Moreover, from \eqref{eq:LownerOrderViaQuadForm} it obviously
follows that if $A_n \preceq B_n$ for any $n \in \mathbb{N}$, and $A_n \to A$ and $B_n \to B$ in $\mathbb{S}^{\ell}$ as
$n \to \infty$, then $A \preceq B$, that is, one can pass to the limit in matrix inequalities.

Denote by $\mathbb{S}^{\ell}_+ = \{ A \in \mathbb{S}^{\ell} \mid A \succeq \mathbb{O}_{\ell \times \ell} \}$ the cone of
all positive semidefinite matrices, where $\mathbb{O}_{n \times m}$ is the zero matrix of order $n \times m$. The cone
of negative semidefinite matrices is denoted by $\mathbb{S}^{\ell}_-$. 
Note that $\mathbb{S}^{\ell}_- = - \mathbb{S}^{\ell}_+$.

\begin{lemma} \label{lem:BoundedOrderBall}
For any bounded set $K \subset \mathbb{S}^{\ell}$ the set 
$\mathbb{B}_K = (K + \mathbb{S}^{\ell}_+) \cap (K + \mathbb{S}^{\ell}_-)$ is bounded.
\end{lemma}

\begin{proof}
Let $C > 0$ be such that $\| X \|_F \le C$ for all $X \in K$. Choose any $X \in \mathbb{B}_K$. Then by definition there
exist $X_1, X_2 \in K$, $Y \in \mathbb{S}^{\ell}_+$, and $Z \in \mathbb{S}^{\ell}_-$ such that 
$X = X_1 + Y = X_2 + Z$. Hence
\[
  \langle X, Z \rangle = \langle X_1, Z \rangle + \langle Y, Z \rangle = \langle X_2, Z \rangle + \| Z \|_F^2.
\]
Note that $\langle Y, Z \rangle \le 0$, since, as is well known and easy to check, $\langle A, B \rangle \le 0$ for any
$A \in \mathbb{S}^{\ell}_+$ and $B \in \mathbb{S}^{\ell}_-$. Consequently, one has
\[
  \| Z \|_F^2 = \langle X_1 - X_2, Z \rangle + \langle Y, Z \rangle 
  \le \| X_1 - X_2 \|_F \| Z \|_F \le 2 C \| Z \|_F.
\]
Therefore, $\| Z \|_F \le 2 C$ and $\| X \|_F = \| X_2 + Z \|_F \le 3 C$, that is, the set $\mathbb{B}_K$ is bounded.
\end{proof}

Let $F \colon \mathbb{R}^d \to \mathbb{S}^{\ell}$ be a given matrix-valued function. Recall that $F$ is called convex
(with respect to the L\"{o}wner partial order), if
\[
  F(\alpha x_1 + (1 - \alpha) x_2) \preceq \alpha F(x_1) + (1 - \alpha) F(x_2)
  \quad \forall x_1, x_2 \in \mathbb{R}^d, \enspace \alpha \in [0, 1].
\]
For any $z \in \mathbb{R}^{\ell}$ denote $F_z(\cdot) = \langle z, F(\cdot) z \rangle$. From the definition of matrix
convexity and relations \eqref{eq:LownerOrderViaQuadForm} it follows that the matrix-valued function $F$ is convex if
and only if for any $z \in \mathbb{R}^{\ell}$ the real-valued function $F_z$ is convex. Choosing as $z$ every vector
from the canonical basis of $\mathbb{R}^{\ell}$ one gets that if $F$ is convex, then its diagonal elements 
$F_{ii}$, $i \in \{ 1, \ldots, \ell \}$, are convex functions. However, non-diagonal elements of convex matrix-valued
functions need not be convex. Namely, the following result holds true (see \cite{Dolgopolik_DC_Semidef_I}).

\begin{theorem} \label{thrm:ConvexImpliesDC_Lipschitz}
Let $F$ be convex. Then for any $i, j \in \{ 1, \ldots, \ell \}$, $i \ne j$, the functions $F_{ii}$ are convex, while
the functions $F_{ij}$ are DC (i.e. they can be represented as the difference of convex functions). Hence, in
particular, the function $F$ is locally Lipschitz continuous and directionally differentiable at every 
$x \in \mathbb{R}^d$.
\end{theorem}

Let us present a simple example illustrating the theorem above.

\begin{example} \label{ex:ConvexFunctWithNonconvexElements}
Let $d = \ell = 2$ and 
\[
  F(x) = \begin{pmatrix} 
	  |x^{(1)}| + |x^{(2)}| & |x^{(1)}| - |x^{(2)}| 
	  \\ 
	  |x^{(1)}| - |x^{(2)}| & |x^{(1)}| + |x^{(2)}|  
	 \end{pmatrix}
  \quad \forall x = (x^{(1)}, x^{(2)}) \in \mathbb{R}^2
\]
Let us check that the function $F$ is convex despite the fact that its non-diagonal elements are nonconvex functions.
Indeed, note that for any $z \in \mathbb{R}^2$ one has
\begin{align*}
  F_z(x) = \langle z, F(x) z \rangle &= \big( |x^{(1)}| + |x^{(2)}| \big) (z^{(1)})^2 
  \\
  &+ 2 \big( |x^{(1)}| - |x^{(2)}| \big) z^{(1)} z^{(2)} + \big( |x^{(1)}| + |x^{(2)}| \big) (z^{(2)})^2 
  \\
  &= |x^{(1)}| \big( z^{(1)} + z^{(2)} \big)^2 + |x^{(2)}| \big( z^{(1)} - z^{(2)} \big)^2.
\end{align*}
The function $F_z(\cdot)$ is obviously convex for any $z \in \mathbb{R}^2$, which implies that the matrix-valued
function $F$ is also convex.
\end{example}

\begin{remark}
Note that for any convex function $f \colon \mathbb{R}^d \to \mathbb{R}$ the matrix-valued function 
$F(\cdot) = \left( \begin{smallmatrix} f(\cdot) & - f(\cdot) \\ - f(\cdot) & f(\cdot) \end{smallmatrix} \right)$ 
is convex, since for any $z \in \mathbb{R}^2$ the function $F_z(\cdot) = f(\cdot) (z^{(1)} - z^{(2)})^2$ is obviously
convex.
\end{remark}

Since non-diagonal elements of a convex matrix-valued function can be nonconvex, natural questions of how to compute
subdifferentials of such functions and whether the classical subdifferential calculus can be extended to the
matrix-valued case arise. We present detailed answers to these questions in the following sections.

\section{General properties of subdifferentials}
\label{sect:GeneralProperties}

Let us start by studying some basic properties of subdifferentials of convex matrix-valued functions. 
Let $F \colon \mathbb{R}^d \to \mathbb{S}^{\ell}$ be a convex function. Recall that the subdifferential of $F$ at a
point $x \in \mathbb{R}^d$ is defined as the set of all those linear operators 
$\mathcal{A} \colon \mathbb{R}^d \to \mathbb{S}^{\ell}$ for which
\[
  F(y) - F(x) \succeq \mathcal{A}(y - x) \quad \forall y \in \mathbb{R}^d.
\]
(see \cite{Papageorgiou,Thera,KusraevKutateladze}). One can readily verify that any linear operator 
$\mathcal{A} \colon \mathbb{R}^d \to \mathbb{S}^{\ell}$ has the form
\[
  \mathcal{A} y = y^{(1)} V^{(1)} + \ldots + y^{(d)} V^{(d)} 
  \quad \forall y = (y^{(1)}, \ldots, y^{(d)}) \in \mathbb{R}^d
\]
for some matrices $V^{(i)} \in \mathbb{S}^{\ell}$, $i \in \{ 1, \ldots, d \}$ (namely, $V^{(i)} = \mathcal{A} e_i$, where 
$e_i$ is the $i$-th vector from the canonical basis of $\mathbb{R}^d$). Therefore, hereinafter we identify linear operators
from $\mathbb{R}^d$ to $\mathbb{S}^{\ell}$ with such collections of matrices and define the subdifferential of $F$ at
a point $x \in \mathbb{R}^d$ as follows:
\begin{multline} \label{eq:SubdiffDef}
  \partial F(x) = \Big\{ (V^{(1)}, \ldots, V^{(d)}) \in \mathbb{S}^{\ell} \times \ldots \times \mathbb{S}^{\ell} \Bigm|
  \\
  F(y) - F(x) \succeq \sum_{i = 1}^d (y^{(i)} - x^{(i)}) V^{(i)} \enspace \forall y \in \mathbb{R}^d \Big\}
\end{multline}
The subdifferential of $F$ shares many properties with subdifferentials of real-valued convex functions.

\begin{theorem} \label{thrm:SubdiffConvexCompact}
For any $x \in \mathbb{R}^d$ the subdifferential of $F$ at $x$ is a compact convex set, and the subdifferential mapping
$\partial F(\cdot)$ is locally bounded.
\end{theorem}

\begin{proof}
Fix any $x \in \mathbb{R}^d$ and let $V, W \in \partial F(x)$. By definition it means that
\[
  F(y) - F(x) \succeq \sum_{i = 1}^d (y^{(i)} - x^{(i)}) V^{(i)}, \quad
  F(y) - F(x) \succeq \sum_{i = 1}^d (y^{(i)} - x^{(i)}) W^{(i)}
\]
for all $y \in \mathbb{R}^d$. Multiplying the first inequality by some $\alpha \in [0, 1]$, the second one by 
$1 - \alpha$, and summing them up one gets
\[
  F(y) - F(x) \succeq \sum_{i = 1}^d (y^{(i)} - x^{(i)}) (\alpha V^{(i)} + (1 - \alpha) W^{(i)})
  \quad \forall y \in \mathbb{R}^d
\]
(see \eqref{eq:LoewnerOrderProp}). Therefore $\alpha V + (1 - \alpha) W \in \partial F(x)$ for any $\alpha \in [0, 1]$,
which means that the subdifferential $\partial F(x)$ is a convex set.

Let now a sequence $\{ V_n \} \subset \partial F(x)$ converge to some $V \in (\mathbb{S}^{\ell})^d$. By the definition of
subdifferential one has
\[
  F(y) - F(x) \succeq \sum_{i = 1}^d (y^{(i)} - x^{(i)}) V_n^{(i)}
  \quad \forall y \in \mathbb{R}^d, \: \forall n \in \mathbb{N}.
\]
Passing to the limit as $n \to \infty$ one obtains that $V \in \partial F(x)$, that is, the subdifferential is closed.
Let us finally show that the subdifferential mapping $\partial F(\cdot)$ is locally bounded. Then, in particular, one
can conclude that $\partial F(x)$ is a compact convex set.

Indeed, introduce the sets $B_{\mathbb{S}^{\ell}} = \{ V \in \mathbb{S}^{\ell} \mid \| V \|_F \le 1 \}$ and 
\[
  \mathbb{B} = 
  \Big( B_{\mathbb{S}^{\ell}} + \mathbb{S}^{\ell}_+ \Big) \cap \Big( B_{\mathbb{S}^{\ell}} + \mathbb{S}^{\ell}_- \Big).
\]
By \cite[Thm.~5]{Robinson76} there exist $r > 0$ and $L > 0$ such that 
\begin{equation} \label{eq:VectorLipshitzCond}
  F(x_1) - F(x_2) \in L \| x_1 - x_2 \| \mathbb{B} \quad \forall x_1, x_2 \in B(x, r), 
\end{equation}
where $B(x, r) = \{ y \in \mathbb{R}^d \mid \| y - x \| \le r \}$ and $\| \cdot \|$ is the Euclidean norm.

Choose any $y \in B(x, 0.5 r)$, $V \in \partial F(y)$, and $i \in \{ 1, \ldots, d \}$. By the definition of
subdifferential for any $z \in \mathbb{R}^d$ one has $F(z) - F(y) \succeq \sum_{i = 1} (z^{(i)} - y^{(i)}) V^{(i)}$,
which by the definition of the L\"{o}wner partial order means that the matrix
$F(z) - F(y) - \sum_{i = 1} (z^{(i)} - y^{(i)}) V^{(i)}$ belongs to $\mathbb{S}^{\ell}_+$. Hence, in particular, 
\[
  (y^{(i)} - t) V^{(i)} \in F(y) - F(y^{(1)}, \ldots, y^{(i-1)}, t, y^{(i + 1)}, \ldots, y^{(d)}) + \mathbb{S}^{\ell}_+
  \quad \forall t \in \mathbb{R}.
\]
Clearly, $y \pm 0.5 r e_i \in B(y, 0.5 r) \subseteq B(x, r)$. Consequently, putting $t = y^{(i)} \pm 0.5 r$ and applying
\eqref{eq:VectorLipshitzCond} one gets that $\pm 0.5 r V^{(i)} \in 0.5 L r \mathbb{B} + \mathbb{S}^{\ell}_+$ or,
equivalently,
\[
  V^{(i)} \in 
  \Big( L \mathbb{B} + \mathbb{S}^{\ell}_+ \Big) \cap \Big( L \mathbb{B} + \mathbb{S}^{\ell}_- \Big)
  \quad \forall i \in \{ 1, \ldots, d \}.
\]
(here we used the fact that $\mathbb{B} = - \mathbb{B}$). The set $\mathbb{B}$ and, consequently, the set on the
right-hand side of the inclusion above are bounded by Lemma~\ref{lem:BoundedOrderBall}. Therefore there exists $R > 0$
(independent of $i \in \{ 1, \ldots, d \}$, $V \in \partial F(y)$, and $y \in B(x, 0.5r)$) such that 
$\| V^{(i)} \|_F \le R$, which implies that the subdifferential mapping $\partial F(\cdot)$ is locally bounded.
\end{proof}

Recall the outer limit of a set-valued mapping $G \colon X \to Y$, with $X$ and $Y$ being metric spaces, at a point 
$x \in X$ is defined as
\begin{multline*}
  \limsup_{x' \to x} G(x') := \Big\{ y \in Y \Bigm| 
  \\
  \exists \{ x_n \} \subset X, \: \{ y_n \} \subset Y
  \colon x_n \to x, \: y_n \to y, \: y_n \in G(x_n) \enspace \forall n \in \mathbb{N} \Big\}
\end{multline*}
(see, e.g. \cite{AubinFrankowska}).

\begin{theorem} \label{thrm:OuterSemiContin}
The subdifferential mapping $\partial F(\cdot)$ is outer semicontinuous (o.s.c.), that is,
\begin{equation} \label{eq:OuterSemiCont}
  \limsup_{y \to x} \partial F(y) \subseteq \partial F(x) \quad \forall x \in \mathbb{R}^d.
\end{equation}
\end{theorem}

\begin{proof}
Fix some $x \in \mathbb{R}^d$. Let some sequence $\{ x_n \} \subset \mathbb{R}^d$ converge to $x$ and a sequence of
subgradients $V_n \in \partial F(x_n)$, $n \in \mathbb{N}$, converge to some $V \in (\mathbb{S}^{\ell})^d$. By the
definition of subdifferential
\[
  F(y) - F(x_n) \succeq \sum_{i = 1}^d (y^{(i)} - x_n^{(i)}) V_n^{(i)} \quad \forall y \in \mathbb{R}^d.
\]
Passing to the limit as $n \to \infty$ with the use of the fact that a convex matrix-valued function is continuous
(see, e.g. \cite[Thm.~5]{Robinson76}) one obtains
\[
  F(y) - F(x) \succeq \sum_{i = 1}^d (y^{(i)} - x^{(i)}) V^{(i)} \quad \forall y \in \mathbb{R}^d,
\]
which obviously means that $V \in \partial F(x)$, that is, inclusion \eqref{eq:OuterSemiCont} holds true.
\end{proof}

\begin{remark}
With the use of the previous theorem and the compactness of subdifferential (or by applying standard results on o.s.c.
compact-valued multifunctions \cite{AubinFrankowska}) one can readily verify that for any $x \in \mathbb{R}^d$ and
$\varepsilon > 0$ there exists $\delta > 0$ such that 
$\partial F(y) \subseteq \partial F(x) + \varepsilon B_{\mathbb{S}^{\ell}}$ for all $y \in B(x, \delta)$.
\end{remark}

In some cases, the following simple characterisation of the subdifferential of $F$ in terms of the subdifferentials of
the functions $F_z(\cdot) = \langle z, F(\cdot) z \rangle$, $z \in \mathbb{R}^{\ell}$, can be useful. Recall that from
the convexity of $F$ it follows that the functions $F_z$ are convex as well.

\begin{lemma} \label{lem:SubdiffQuadFormCharacterization}
For any $x \in \mathbb{R}^d$ the following equality holds true:
\[
  \partial F(x) = \Big\{ V \in (\mathbb{S}^{\ell})^d \Bigm|
  (\langle z, V^{(1)} z \rangle, \ldots, \langle z, V^{(d)} z \rangle) \in \partial F_z(x) \enspace 
  \forall z \in \mathbb{R}^{\ell} \Big\}.
\]
\end{lemma}

\begin{proof}
By definition $V \in \partial F(x)$ if and only if
\[
  \langle z, \Big( F(y) - F(x) - \sum_{i = 1}^d (y^{(i)} - x^{(i)}) V^{(i)} \Big) z \rangle \ge 0
  \quad \forall z \in \mathbb{R}^{\ell} \: \forall y \in \mathbb{R}^d
\]
(see \eqref{eq:SubdiffDef} and \eqref{eq:LownerOrderViaQuadForm}) or, equivalently,
\[
  F_z(y) - F_z(x) \ge \sum_{i = 1}^d \langle z, V^{(i)} z \rangle (y^{(i)} - x^{(i)})
  \quad \forall z \in \mathbb{R}^{\ell} \: \forall y \in \mathbb{R}^d,
\]
which is equivalent to the inclusion 
$(\langle z, V^{(1)} z \rangle, \ldots, \langle z, V^{(d)} z \rangle) \in \partial F_z(x)$
for all $z \in \mathbb{R}^{\ell}$.
\end{proof}

With the use of the previous lemma we can show that the subdifferential of a Fr\'{e}chet differentiable convex
matrix-valued function is, as one can expect, a singleton.

\begin{proposition} \label{prp:SubdiffAtDiffPoint}
If $F$ is G\^{a}teaux differentiable at a point $x$, then $F'(x) \in \partial F(x)$, where $F'(x)$ is the G\^{a}teaux
derivative of $F$ at $x$. If $F$ is Fr\'{e}chet differentiable at $x$, then $\partial F(x) = \{ F'(x) \}$.
\end{proposition}

\begin{proof}
The fact that $F'(x) \in \partial F(x)$ follows, in particular, from \cite[Lemma~1]{Dolgopolik_DC_Semidef_I}. Let us
show that the subdifferential is a singleton in the case when $F$ is Fr\'{e}chet differentiable at $x$. 

Indeed, let $V \in \partial F(x)$. Applying Lemma~\ref{lem:SubdiffQuadFormCharacterization} with $z$ running through the
canonical basis of $\mathbb{R}^{\ell}$ one gets that $(V_{ii}^{(1)}, \ldots, V_{ii}^{(d)}) = \nabla F_{ii}(x)$ for all 
$i \in \{ 1, \ldots, \ell \}$. Hence choosing any $i, j \in \{ 1, \ldots, \ell \}$, $i \ne j$, defining 
$z^{(i)} = z^{(j)} = 1$ and $z^{(k)} = 0$, if $k \notin \{ i, j \}$, and applying
Lemma~\ref{lem:SubdiffQuadFormCharacterization} once again, one obtains that
\begin{multline*}
  (V_{ii}^{(1)} + V_{jj}^{(1)} + 2 V_{ij}^{(1)}, \ldots, V_{ii}^{(d)} + V_{jj}^{(d)} + 2 V_{ij}^{(d)})
  \\
  = \left( \frac{\partial F_{ii}(x)}{\partial x^{(1)}} + \frac{\partial F_{jj}(x)}{\partial x^{(1)}} + 2 V_{ij}^{(1)},
  \ldots,
  \frac{\partial F_{ii}(x)}{\partial x^{(d)}} + \frac{\partial F_{jj}(x)}{\partial x^{(d)}} + 2 V_{ij}^{(d)} \right)
  \\
  \in \partial (F_{ii} + F_{jj} + 2 F_{ij})(x) = \nabla F_{ii}(x) + \nabla F_{jj}(x) + 2 \nabla F_{ij}(x).
\end{multline*}
Here we used the fact that $\partial (F_{ii} + F_{jj} + 2 F_{ij})(x) = \nabla (F_{ii} + F_{jj} + 2 F_{ij})(x)$, since
the function $F_{ii} + F_{jj} + 2 F_{ij}$ is Fr\'{e}chet differentiable at $x$. Thus, one can conclude that
$(V_{ij}^{(1)}, \ldots, V_{ij}^{(d)}) = \nabla F_{ij}(x)$ for any $i, j \in \{ 1, \ldots, d \}$.
Consequently, $V^{(i)} = \partial F(x)/\partial x^{(i)}$ for any $i \in \{ 1, \ldots, d \}$ and $V = F'(x)$.
\end{proof}

\section{Computing subdifferentials: univariate case}
\label{sect:OneDimCase}

In this section we study subdifferentials of convex matrix-valued functions in the one-dimensional case, that is, in the
case of functions defined on the real line $\mathbb{R}$. Let a function $F \colon \mathbb{R} \to \mathbb{S}^{\ell}$ be
given. First, we provide a convenient (from the theoretical point of view) characterisation of the convexity of $F$, which
is an extension of a well-known result for real-valued convex functions to the matrix-valued case (see, e.g.
\cite[Section~24]{Rockafellar}).

\begin{lemma}
The function $F$ is convex if and only if for any $x_1 < x < x_2$ one has
\begin{equation} \label{eq:IncreasingSecant}
  \frac{1}{x - x_1} \big( F(x) - F(x_1) \big) \preceq \frac{1}{x_2 - x} \big( F(x_2) - F(x) \big)
\end{equation}
\end{lemma}

\begin{proof}
Let $F$ be convex. Choose any points $x_1 < x < x_2$. Observe that $x$ can be represented as 
$x = \alpha x_2 + (1 - \alpha) x_1$ with $\alpha = (x - x_1) / (x_2 - x_1) \in (0, 1)$. Therefore by the definition of
convexity one has
\[
  F(x) = F(\alpha x_2 + (1 - \alpha) x_1) \preceq \alpha F(x_2) + (1 - \alpha) F(x_1)
\]
or, equivalently,
\[
  \left( \frac{x - x_1}{x_2 - x_1} + \frac{x_2 - x}{x_2 - x_1} \right) F(x)
  \preceq \frac{x - x_1}{x_2 - x_1} F(x_2) + \frac{x_2 - x}{x_2 - x_1} F(x_1).
\]
Multiplying this inequality by $x_2 - x_1$, then dividing it by $(x - x_1)(x_2 - x)$ and rearranging the terms with the use 
of \eqref{eq:LoewnerOrderProp} one gets that inequality \eqref{eq:IncreasingSecant} holds true.

Suppose now that for any $x_1 < x < x_2$ inequality \eqref{eq:IncreasingSecant} holds true. Choose any 
$x_1, x_2 \in \mathbb{R}$. If $x_1 = x_2$, then 
$F(\alpha x_1 + (1 - \alpha) x_2) = \alpha F(x_1) + (1 - \alpha) F(x_2)$ for any $\alpha \in [0, 1]$.
Therefore, assume that $x_1 < x_2$ (the case $x_1 > x_2$ can be considered in exactly the same way). 

Fix any $\alpha \in (0, 1)$ and define $x = \alpha x_1 + (1 - \alpha) x_2$. Then $x_1 < x < x_2$ and by inequality
\eqref{eq:IncreasingSecant} one has
\[
  \frac{1}{(1 - \alpha)(x_2 - x_1)} \big( F(x) - F(x_1) \big) 
  \preceq \frac{1}{\alpha(x_2 - x_1)} \big( F(x_2) - F(x) \big).
\]
Multiplying this inequality by $\alpha (1 - \alpha) (x_2 - x_1)$ and rearranging the terms with the use of
\eqref{eq:LoewnerOrderProp} one obtains that $F(x) \preceq \alpha F(x_1) + (1 - \alpha) F(x_2)$. Hence taking into
account the facts that $\alpha \in (0, 1)$ was chosen arbitrarily and $x = \alpha x_1 + (1 - \alpha) x_2$ one can
conclude that the function $F$ is convex.
\end{proof}

If $F$ is convex, then by Theorem~\ref{thrm:ConvexImpliesDC_Lipschitz} it is directionally differentiable, which
implies that for any $x \in \mathbb{R}$ there exist the right-hand $F'_+(x)$ and the left-hand $F'_-(x)$ derivatives of
$F$ at $x$.

\begin{corollary} \label{crlr:ConvexityVsDD}
If $F$ is convex, then for any $x_1 < x_2$ and $x \in \mathbb{R}$ one has
\begin{gather} \label{eq:DDInequal}
  F'_+(x_1) \preceq \frac{1}{x_2 - x_1} \big( F(x_2) - F(x_1) \big) \preceq F'_-(x_2)
  \\ \label{eq:DDMonotone}
  F'_-(x) \preceq F'_+(x).
\end{gather}
In particular, both mappings $F'_+(\cdot)$ and $F'_-(\cdot)$ are nondecreasing with respect to the L\"{o}wner partial
order.
\end{corollary}

\begin{proof}
Passing to the limit in \eqref{eq:IncreasingSecant} as $x$ tends to $x_2$ from the left and then as $x$ tends to $x_1$
from the right one obtains that inequalities \eqref{eq:DDInequal} are valid. In turn, passing to the limit in
\eqref{eq:IncreasingSecant} as $x_1$ tends to $x$ from the left and then as $x_2$ tends to $x$ from the right one gets
that inequality \eqref{eq:DDMonotone} holds true.
\end{proof}

With the use of Corollary~\ref{crlr:ConvexityVsDD} we can provide a simple characterisation of subdifferentials of
convex matrix-valued functions defined on $\mathbb{R}$.

\begin{theorem} \label{thrm:OneDimCase}
Let $F$ be convex. Then for any $x \in \mathbb{R}$ the subdifferential of $F$ at $x$ is nonempty and satisfies the
following equality:
\begin{equation} \label{eq:SubdiffViaDD}
  \partial F(x) = \Big\{ V \in \mathbb{S}^{\ell} \Bigm| F'_-(x) \preceq V \preceq F'_+(x) \Big\}.
\end{equation}
In particular, $F'_+(x) \in \partial F(x)$ and $F'_-(x) \in \partial F(x)$. 
\end{theorem}

\begin{proof}
Choose any $x, y \in \mathbb{R}$. If $y > x$, then applying the first inequality in \eqref{eq:DDInequal} with $x_1 = x$
and $x_2 = y$ one gets
\begin{equation} \label{eq:RHsideDerSubgrad}
  F(y) - F(x) \succeq (y - x) F'_+(x).
\end{equation}
In turn, if $y < x$, then applying the second inequality in \eqref{eq:DDInequal} with $x_2 = x$ and  $x_1 = y$ and
inequality 
\eqref{eq:DDMonotone} one obtains
\begin{equation} \label{eq:LH_RHDerSubgrad}
  \frac{1}{x - y} \big( F(x) - F(y) \big) \preceq F'_-(x) \preceq F'_+(x)
\end{equation}
Multiplying this inequality by $x - y$ and rearranging the terms one gets that inequality \eqref{eq:RHsideDerSubgrad}
holds true. Thus, inequality \eqref{eq:RHsideDerSubgrad} is satisfied for all $y \in \mathbb{R}$, which implies that the
subdifferential $\partial F(x)$ is nonempty and $F'_+(x) \in \partial F(x)$.

From inequalities \eqref{eq:RHsideDerSubgrad} and \eqref{eq:DDMonotone} it follows that 
$F(y) - F(x) \succeq (y - x) F'_-(x)$ in the case $y > x$. The validity of this inequality in the case $y < x$ follows
directly from the first inequality in \eqref{eq:LH_RHDerSubgrad}. Consequently, $F'_-(x) \in \partial F(x)$.

Let us now prove equality \eqref{eq:SubdiffViaDD}. Indeed, fix any $V \in \partial F(x)$. By definition
$F(y) - F(x) \succeq (y - x) V$ for any $y \in \mathbb{R}$. Choosing $y > x$, dividing this inequality by $y - x$, and
passing to the limit as $y$ tends to $x$ from the right one gets that $F'_+(x) \succeq V$. Similarly, choosing $y < x$,
dividing by $y - x$, and passing to the limit as $y$ tends to $x$ from the left one obtains $V \succeq F'_-(x)$, which
implies that
\[
  \partial F(x) \subseteq \Big\{ V \in \mathbb{S}^{\ell} \Bigm| F'_-(x) \preceq V \preceq F'_+(x) \Big\}. 
\]
Let us prove the converse inclusion. Choose any $V$ from the set on the right-hand of the inclusion above. Then for
any $y > x$ one has $(y - x) F'_+(x) \succeq (y - x) V$. Hence with the use of the definition of subgradient and the
fact that $F'_+(x) \in \partial F(x)$ one gets that for any $y > x$ the following inequality holds true:
\begin{equation} \label{eq:OneDimSubgrad}
  F(y) - F(x) \succeq (y - x) V
\end{equation}
In turn, for any $y < x$ one has $(y - x) V \preceq (y - x) F'_-(x) \preceq F(y) - F(x)$, since 
$F'_-(x) \in \partial F(x)$, which means that inequality \eqref{eq:OneDimSubgrad} is satisfied for any $y < x$
and $V \in \partial F(x)$.
\end{proof}

\begin{remark}
The previous theorem provides a simple rule for computing subgradients of convex functions 
$G \colon \mathbb{R} \to \mathbb{S}^{\ell}$. Namely, one simply needs to compute, say, right-hand derivatives of each
component $G_{ij}$ of $G$ at a given point $x$ and define $V = G'_+(x)$. Then $V \in \partial G(x)$.
\end{remark}

Corollary~\ref{crlr:ConvexityVsDD} can also be used to reveal some interesting differentiability properties of a convex
matrix-valued function. In particular, this corollary implies that if all diagonal elements of such function are
differentiable, then this function is differentiable as well.

\begin{theorem} \label{thrm:Smoothness}
Let $F$ be convex and for some $i \in \{ 1, \ldots, \ell \}$ the function $F_{ii}$ be differentiable at a point 
$x \in \mathbb{R}$. Then for all $j \in \{ 1, \ldots, \ell \}$ the functions $F_{ij} = F_{ji}$ are differentiable at $x$
as well.
\end{theorem}

\begin{proof}
Denote $A = F'_+(x)$ and $B = F'_-(x)$. By our assumption $A_{ii} = B_{ii}$. Our aim is to show that $A_{ij} = B_{ij}$
for all $j \in \{ 1, \ldots, \ell \}$. Then one can conclude that the corresponding functions $F_{ij}$ are
differentiable at $x$.

Suppose by contradiction that $A_{ij} \ne B_{ij}$ for some $j \in \{ 1, \ldots, \ell \}$. Note that $A \succeq B$, i.e.
the matrix $A - B$ is positive semidefinite, by Corollary~\ref{crlr:ConvexityVsDD}. Consequently, for any vector 
$z \in \mathbb{R}^{\ell}$ with $z_k = 0$ for all $k \in \{ 1, \ldots, \ell \} \setminus \{ i, j \}$ one has
\[
  0 \le \langle z, (A - B) z \rangle = 2 (A_{ij} - B_{ij}) z_i z_j + (A_{jj} - B_{jj}) z_j^2.
\]
Hence for $z_j = 1$ and $z_i = t \sign(A_{ij} - B_{ij})$ one has $2 |A_{ij} - B_{ij}| t + A_{jj} - B_{jj} \ge 0$ for
all $t \in \mathbb{R}$, which is obviously impossible.
\end{proof}

\begin{corollary} \label{crlr:Smoothness}
Let $F$ be convex and for all $i \in \{ 1, \ldots, \ell \}$ the functions $F_{ii}$ be differentiable at a point 
$x \in \mathbb{R}$. Then $F$ is differentiable at $x$.
\end{corollary}

Thus, points of nonsmoothness of the diagonal elements of a convex matrix-valued function must be consistent
with points of nonsmoothness of non-diagonal elements from the corresponding row and column. For example, the
function $F \colon \mathbb{R} \to \mathbb{S}^2$, 
$F(x) = \left( \begin{smallmatrix} F_{11}(x) & |x| \\ |x| & F_{22}(x) \end{smallmatrix} \right)$ 
cannot be convex, if either of the functions $F_{11}$ or $F_{22}$ are differentiable at zero.

\begin{remark}
Theorem~\ref{thrm:Smoothness} and Corollary~\ref{crlr:Smoothness} admit obvious extensions to the case of convex
matrix-valued functions defined on $\mathbb{R}^d$. In particular, let $F \colon \mathbb{R}^d \to \mathbb{S}^{\ell}$ be a
convex function and for all $i \in \{ 1, \ldots, \ell \}$ the functions $F_{ii}$ be differentiable in $x^{(k)}$ at a
point $x \in \mathbb{R}^d$ for some $k \in \{ 1, \ldots, d \}$. Then applying Corollary~\ref{crlr:Smoothness} to the
mapping $x^{(k)} \mapsto F(x)$ one obtains that $F$ is differentiable in $x^{(k)}$ as well.
\end{remark}

\section{Computing subdifferentials: multivariate case}
\label{sect:MultiDimCase}

Let us now consider the multi-dimensional case. Suppose that a convex function 
$F \colon \mathbb{R}^d \to \mathbb{S}^{\ell}$ is given. By Theorem~\ref{thrm:ConvexImpliesDC_Lipschitz} the function $F$
is locally Lipschitz continuous. Therefore, by Rademacher's theorem $F$ is almost everywhere Fr\'{e}chet differentiable
and one can define its Clarke subdifferential (generalised Jacobian) as follows:
\[
  \partial_{Cl} F(x) = \co\Big\{ V \in (\mathbb{S}^{\ell})^d \Bigm| \exists \{ x_n \} \subset \Omega_F \colon
  \lim_{n \to \infty} x_n = x, \: \lim_{n \to \infty} F'(x_n) = V \Big\},
\]
where $\Omega_F$ is the set of points $x \in \mathbb{R}^d$ at which $F$ is Fr\'{e}chet differentiable. From the general
properties of the Clarke subdifferentials/generalised Jacobians of locally Lipschitz mappings (see \cite{Clarke}) it
follows that the Clarke subdifferential of a convex matrix-valued function possesses the following properties. 

\begin{proposition}
For any $x \in \mathbb{R}^d$ the Clarke subdifferential $\partial_{Cl} F(x)$ is correctly defined, nonempty, convex, and
compact.
\end{proposition}

With the use of this proposition we can prove that the subdifferential of a convex matrix-valued function is always
nonempty and one can compute its subgradients in the sense of convex analysis by computing its Clarke subgradients.

\begin{theorem} \label{thrm:ConvAnalSubdiff_vs_ClarkSubdiff}
For any $x \in \mathbb{R}^d$ the subdifferential of $F$ at $x$ in the sense of convex analysis is not empty and
$\partial_{Cl} F(x) \subseteq \partial F(x)$.
\end{theorem}

\begin{proof}
Let $V \in \partial_{Cl} F(x)$. Then by definition there exists a sequence $\{ x_n \} \subset \mathbb{R}^d$ converging
to $x$ and such that $F$ is Fr\'{e}chet differentiable at each point $x_n$ and $F'(x_n) \to V$ as $n \to \infty$. By
Proposition~\ref{prp:SubdiffAtDiffPoint} $F'(x_n) \in \partial F(x_n)$, which due to the outer semicontinuity of the
subdifferential map (Theorem~\ref{thrm:OuterSemiContin}) implies that $V \in \partial F(x)$. Thus, the subdifferential
$\partial F(x)$ is nonempty and $\partial_{Cl} F(x) \subseteq \partial F(x)$.
\end{proof}

As we will show in the following section, the inclusion from the theorem above can be strict. Nonetheless, it allows
one to compute subgradients of convex matrix-valued functions by using standard rules for computing generalised
Jacobians \cite{Clarke}.

\section{Subdifferential calculus}
\label{sect:SubdiffCalc}

Let us finally present some simple calculus rules for subdifferentials of convex matrix-valued functions. 
Let $F \colon \mathbb{R}^d \to \mathbb{S}^{\ell}$ be a convex function. First, note that for any 
$A \in \mathbb{S}^{\ell}$ one obviously has $\partial (F + A)(\cdot) = \partial F(\cdot)$. Next we study how
subdifferentials change under matrix transformations.

\begin{theorem}
For any matrix $M \in \mathbb{R}^{m \times \ell}$ the function $F_M(\cdot) := M F(\cdot) M^T$ is convex and 
\begin{equation} \label{eq:MatrixTransformSubdiff}
  M \partial F(\cdot) M^T
  := \Big\{ (M V^{(1)} M^T, \ldots M V^{(d)} M^T) \in (\mathbb{S}^m)^d \Bigm| V \in \partial F(\cdot) \Big\}
  \subseteq \partial F_M(\cdot).
\end{equation}
Moreover, if $m = \ell$ and the matrix $M$ is invertible, then this inclusion is satisfied as equality.
\end{theorem}

\begin{proof}
By the definition of convexity and the L\"{o}wner partial order for any $x_1, x_2 \in \mathbb{R}^d$ and 
$\alpha \in [0, 1]$ one has
\[
  \langle z, \Big(\alpha F(x_1) + (1 - \alpha) F(x_2) - F(\alpha x_1 + (1 - \alpha) x_2)\Big) z \rangle \ge 0 
  \quad \forall z \in \mathbb{R}^{\ell}.
\]
Consequently, for any $h \in \mathbb{R}^m$ one has
\begin{multline*}
  \langle h, \Big(\alpha F_M(x_1) + (1 - \alpha) F_M(x_2) - F_M(\alpha x_1 + (1 - \alpha) x_2)\Big) h \rangle
  \\
  = \langle h, M \Big( \alpha F(x_1) + (1 - \alpha) F(x_2) - F(\alpha x_1 + (1 - \alpha) x_2) \Big) M^T h \rangle
  \\
  = \langle z, \Big(\alpha F(x_1) + (1 - \alpha) F(x_2) - F(\alpha x_1 + (1 - \alpha) x_2)\Big) z \rangle \ge 0
\end{multline*}
where $z = M^T h$. In other words, for any $x_1, x_2 \in \mathbb{R}^d$ and $\alpha \in [0, 1]$ one has
$F_M(\alpha x_1 + (1 - \alpha) x_2) \preceq \alpha F_M(x_1) + (1 - \alpha) F_M(x_2)$, i.e. the function $F_M$ is convex.

Fix any $x \in \mathbb{R}^d$ and $V \in \partial F(x)$. By the definition of subdifferential and the L\"{o}wner partial
order one has
\begin{equation} \label{eq:SubgradDirectDef}
  \langle z, \Big( F(y) - F(x) - \sum_{i = 1}^d (y^{(i)} - x^{(i)}) V^{(i)} \Big) z \rangle \ge 0 
  \quad \forall y \in \mathbb{R}^d, \: z \in \mathbb{R}^{\ell}.
\end{equation}
Therefore, for any $y \in \mathbb{R}^d$ and $h \in \mathbb{R}^m$ the following inequalities hold true: 
\begin{multline*}
  \langle h, \Big( F_M(y) - F_M(x) - \sum_{i = 1}^d (y^{(i)} - x^{(i)}) M V^{(i)} M^T \Big) h \rangle 
  \\
  = \langle h, M \Big( F(y) - F(x) - \sum_{i = 1}^d (y^{(i)} - x^{(i)}) V^{(i)} \Big) M^T h \rangle
  \\
  = \langle z, \Big( F(y) - F(x) - \sum_{i = 1}^d (y^{(i)} - x^{(i)}) V^{(i)} \Big) z \rangle \ge 0
\end{multline*}
where $z = M^T h$. Thus, one has $F_M(y) - F_M(x) \succeq \sum_{i = 1}^d (y^{(i)} - x^{(i)}) M V^{(i)} M^T$ for any 
$y \in \mathbb{R}^d$ or, equivalently, $(M V^{(1)} M^T, \ldots M V^{(d)} M^T) \in \partial F_M(x)$, that is, inclusion
\eqref{eq:MatrixTransformSubdiff} is valid. 

Suppose, finally, that $m = \ell$ and $M$ is invertible. Choose any $W \in \partial F_M(x)$. By definitions
\begin{equation} \label{eq:MatrixTransformSubgradDirectDef}
  \langle h, \Big( F_M(y) - F_M(x) - \sum_{i = 1}^d (y^{(i)} - x^{(i)}) W^{(i)} \Big) h \rangle \ge 0 
  \quad \forall h \in \mathbb{R}^m, \: y \in \mathbb{R}^d.
\end{equation}
For each $i \in \{ 1, \ldots, d \}$ define $V^{(i)} = M^{-1} W^{(i)} (M^T)^{-1}$. Then $W^{(i)} = M V^{(i)} M^T$ and the
inequality above implies that
\begin{align*}
  0 &\le \langle h, M \Big( F(y) - F(x) - \sum_{i = 1}^d (y^{(i)} - x^{(i)}) V^{(i)} \Big) M^T h \rangle
  \\
  &= \langle M^T h, \Big( F(y) - F(x) - \sum_{i = 1}^d (y^{(i)} - x^{(i)}) V^{(i)} \Big) M^T h \rangle
\end{align*}
for any $y \in \mathbb{R}^d$ and $h \in \mathbb{R}^m$. Since $M$ is invertible, for any $z \in \mathbb{R}^{\ell}$ there
exists $h \in \mathbb{R}^m$ such that $z = M^T h$. Consequently, inequality \eqref{eq:SubgradDirectDef} is satisfied or,
equivalently, $V \in \partial F(x)$. Hence bearing in mind the fact that $W = M V M^T$ one can conclude that the
inclusion opposite to \eqref{eq:MatrixTransformSubdiff} holds true, which completes the proof.
\end{proof}

By putting $M = \sqrt{\alpha} I_{\ell}$ with $\alpha > 0$ (the case $\alpha = 0$ is trivial), where $I_{\ell}$ is the
identity matrix of order $\ell$, one obtains the following result.

\begin{corollary}
For any $\alpha \ge 0$ one has $\partial (\alpha F)(\cdot) = \alpha \partial F(\cdot)$.
\end{corollary}

Let us also consider the element-wise (Hadamard) product, denoted by $\odot$, with a constant matrix.

\begin{proposition}
For any positive semidefinite matrix $M \in \mathbb{S}^{\ell}$ the function $G_M(\cdot) = M \odot F(\cdot)$ is convex 
and
\begin{equation} \label{eq:HadamardProductSubdiff}
  M \odot \partial F(\cdot) 
  := \Big\{ (M \odot V^{(1)}, \ldots M \odot V^{(d)}) \in (\mathbb{S}^{\ell})^d \Bigm| V \in \partial F(\cdot) \Big\}
  \subseteq \partial G_M(\cdot)
\end{equation}
If, in addition, all elements of the matrix $M$ are nonzero and the matrix $N \in \mathbb{S}^{\ell}$ with elements 
$N_{ij} = 1 / M_{ij}$ is positive semidefinite, then this inclusion is satisfied as equality.
\end{proposition}

\begin{proof}
From the convexity of $F$ it follows that for any $x_1, x_2 \in \mathbb{R}^d$ and $\alpha \in [0, 1]$ one has
\[
  \alpha F(x_1) + (1 - \alpha) F(x_2) - F(\alpha x_1 + (1 - \alpha) x_2) \succeq \mathbb{O}_{\ell \times \ell},
\]
that is, the matrix $\alpha F(x_1) + (1 - \alpha) F(x_2) - F(\alpha x_1 + (1 - \alpha) x_2)$ is positive semidefinite. 
By the Schur product theorem \cite[Theorem~7.5.3]{HornJohnson} the element-wise product of two positive semidefinite
matrices is positive semidefinite. Therefore
\[
  M \odot \Big( \alpha F(x_1) + (1 - \alpha) F(x_2) - F(\alpha x_1 + (1 - \alpha) x_2) \Big) 
  \succeq \mathbb{O}_{\ell \times \ell}
\]
or, equivalently,
\[
  M \odot F(\alpha x_1 + (1 - \alpha) x_2) \preceq \alpha M \odot F(x_1) + (1 - \alpha) M \odot F(x_2)
\]
for any $x_1, x_2 \in \mathbb{R}^d$ and $\alpha \in [0, 1]$, which means that the function $M \odot F$ is convex.

Fix any $x \in \mathbb{R}^d$ and $V \in \partial F(x)$. By the definition of subdifferential
\[
  \mathbb{O}_{\ell \times \ell} \preceq F(y) - F(x) - \sum_{i = 1}^d (y^{(i)} - x^{(i)}) V^{(i)}
  \quad \forall y \in \mathbb{R}^d.
\]
Hence by the Schur product theorem for all $y \in \mathbb{R}^d$ one has
\begin{align*}
  \mathbb{O}_{\ell \times \ell} &\preceq M \odot \Big( F(y) - F(x) - \sum_{i = 1}^d (y^{(i)} - x^{(i)}) V^{(i)} \Big)
  \\
  &= M \odot F(y) - M \odot F(x) - \sum_{i = 1}^d (y^{(i)} - x^{(i)}) M \odot V^{(i)},
\end{align*}
which by the definition of subdifferential means that $(M\odot V^{(1)}, \ldots M \odot V^{(d)}) \in \partial G_M(x)$ 
or, equivalently, inclusion \eqref{eq:HadamardProductSubdiff} holds true.

Suppose finally that all elements of the matrix $M$ are nonzero and the matrix $N$ with $N_{ij} = M_{ij}^{-1}$ is
positive semidefinite. Choose any $V \in \partial G_M(x)$. By definition
\[
  \mathbb{O}_{\ell \times \ell} \preceq M \odot F(y) - M \odot F(x) - \sum_{i = 1}^d (y^{(i)} - x^{(i)}) V^{(i)}
  \quad \forall y \in \mathbb{R}^d.
\]
Define $W = N \odot V$, that is, $W^{(i)} = N \odot V^{(i)}$ for all $i \in \{ 1, \ldots, d \}$. With the use of the 
Schur product theorem and the fact that $N \odot M \odot A = A$ for any $A \in \mathbb{S}^{\ell}$ one obtains that
\begin{align*}
  \mathbb{O}_{\ell \times \ell} 
  &\preceq N \odot \Big( M \odot F(y) - M \odot F(x) - \sum_{i = 1}^d (y^{(i)} - x^{(i)}) V^{(i)} \Big) 
  \\
  &= F(y) - F(x) - \sum_{i = 1}^d (y^{(i)} - x^{(i)}) W^{(i)}
\end{align*}
for any $y \in \mathbb{R}^d$, which by definition means that $W \in \partial F(x)$. It is easily seen that 
$V = M \odot N \odot V = M \odot W$. Thus, for an arbitrary $V \in \partial G_M(x)$ we have found $W \in \partial F(x)$
such that $V = M \odot W$. In other words, the inclusion opposite to \eqref{eq:HadamardProductSubdiff} holds true.
\end{proof}

Next we consider the subdifferential of the sum of convex functions.

\begin{proposition} \label{prp:SumSubdiff}
Let $F_1, F_2 \colon \mathbb{R}^d \to \mathbb{S}^{\ell}$ be convex functions. Then the function $F_1 + F_2$ is convex
and $\partial F_1(\cdot) + \partial F_2(\cdot) \subseteq \partial (F_1 + F_2)(\cdot)$. If, in addition, one of 
the functions $F_1$ or $F_2$ is Fr\'{e}chet differentiable at some point $x \in \mathbb{R}^d$, then 
$\partial F_1(x) + \partial F_2(x) = \partial (F_1 + F_2)(x)$.
\end{proposition}

\begin{proof}
The convexity of the function $F = F_1 + F_2$ can be readily verified directly. Let us prove the inclusion for the
subdifferentials. To this end, fix any $x \in \mathbb{R}^d$ and $V_k \in \partial F_k(x)$, $k \in \{ 1, 2 \}$. By
definition
\[
  \langle z, \Big( F_k(y) - F_k(x) - \sum_{i = 1}^d (y^{(i)} - x^{(i)}) V_k^{(i)} \Big) z \rangle \ge 0 
  \quad \forall z \in \mathbb{R}^{\ell} \: \forall y \in \mathbb{R}^d.
\]
Summing up these two inequalities one gets
\[
  \langle z, \Big( F(y) - F(x) - \sum_{i = 1}^d (y^{(i)} - x^{(i)}) (V_1^{(i)} + V_2^{(i)}) \Big) z \rangle \ge 0 
  \quad \forall z \in \mathbb{R}^{\ell} \: \forall y \in \mathbb{R}^d,
\]
which means that $V_1 + V_2 \in \partial F(x)$, that is 
$\partial F_1(x) + \partial F_2(x) \subseteq \partial (F_1 + F_2)(x)$.

Suppose now that, say, the function $F_2$ is Fr\'{e}chet differentiable at $x$ and fix any 
$V \in \partial(F_1 + F_2)(x)$. By Lemma~\ref{lem:SubdiffQuadFormCharacterization} for any $z \in \mathbb{R}^{\ell}$ one
has
\[
  \big( \langle z, V^{(1)} z \rangle, \ldots, \langle z, V^{(d)} z \rangle \big) \in \partial (F_1 + F_2)_z(x)
  = \partial (F_1)_z(x) + \partial (F_2)_z(x).
\]
Hence bearing in mind the fact that
\[
  \partial (F_2)_z(x) = \big( \langle z, F'_2(x)^{(1)} z \rangle, \ldots, \langle z, F'_2(x)^{(d)} z \rangle \big)
\]
for any $z \in \mathbb{R}^{\ell}$ by Proposition~\ref{prp:SubdiffAtDiffPoint} one gets that
\[
  \big( \langle z, (V - F'_2(x))^{(1)} z \rangle, \ldots, \langle z, (V - F'_2(x))^{(d)} z \rangle \big)
  \in \partial (F_1)_z(x)
\]
for any $z \in \mathbb{R}^{\ell}$. Consequently, $V - F'_2(x) \in \partial F_1(x)$ by 
Lemma~\ref{lem:SubdiffQuadFormCharacterization}. Thus, one has 
$\partial (F_1 + F_2)(x) \subseteq \partial F_1(x) + F'_2(x)$, which implies the required result.
\end{proof}

\begin{remark}
Note that unlike the (fuzzy) sum rules for various generalisations of subdifferential in the sense of convex analysis to
the nonsmooth \textit{nonconvex} case (see \cite{Mordukhovich_I,Ioffe2012,Penot_book}), the sum rule from the previous
proposition can be used to compute subgradients of nonsmooth convex matrix-valued functions in practice. Namely, if 
$V_1 \in \partial F_1(x)$ and $V_2 \in \partial F_2(x)$ are computed, then the previous proposition guarantees that 
$V_1 + V_2 \in \partial (F_1 + F_2)(x)$.
\end{remark}

\begin{proposition}
Let $G(y) = F(A y + b)$, $y \in \mathbb{R}^m$, for some matrix $A \in \mathbb{R}^{d \times m}$ and vector 
$b \in \mathbb{R}^d$. Then the function $G$ is convex and for any $y \in \mathbb{R}^s$ one has
\begin{multline*}
  A^* \partial F(Ay + b) :=
  \\
  \Big\{ \Big( \sum_{i = 1}^d A_{i1} V^{(i)}, \ldots, \sum_{i = 1}^d A_{im} V^{(i)} \Big) 
  \in (\mathbb{S}^{\ell})^m \Bigm| 
  (V^{(1)}, \ldots, V^{(d)}) \in \partial F(Ay + b) \Big\} 
  \\
  \subseteq \partial G(y).
\end{multline*}
\end{proposition}

\begin{proof}
The convexity of the function $G$ can be readily verified directly. Let us prove the inclusion for the subdifferential.

Fix any $y \in \mathbb{R}^m$, define $x = A y + b$, and choose any $V \in \partial F(x)$. By the definition of
subdifferential for any $u \in \mathbb{R}^m$ and $z \in \mathbb{R}^{\ell}$ one has
\begin{align*}
  0 &\le \langle z, \Big( F(x(u)) - F(x) - \sum_{i = 1}^d (x(u)^{(i)} - x^{(i)}) V^{(i)} \Big) z \rangle
  \\
  &= \langle z, \Big( G(u) - G(y) 
  - \sum_{i = 1}^d \Big[ \sum_{j = 1}^m A_{ij} (u^{(j)} - y^{(j)}) \Big] V^{(i)} \Big) z \rangle
  \\
  &= \langle z, \Big( G(u) - G(y) - \sum_{j = 1}^m (u^{(j)} - y^{(j)}) W^{(j)} \Big) z \rangle
\end{align*}
where $x(u) = A u + b$ and $W^{(j)} = \sum_{i = 1}^d A_{ij} V^{(i)}$. Consequently, by the definition of subgradient
$W = (W^{(1)}, \ldots, W^{(m)}) \in \partial G(y)$, which completes the proof.
\end{proof}

Let us finally consider the subdifferential of a block matrix function.

\begin{theorem} \label{thrm:BlockMatrixSubdiff}
Let $F_1 \colon \mathbb{R}^d \to \mathbb{S}^{\ell}$ and $F_2 \colon \mathbb{R}^d \to \mathbb{S}^m$ be convex functions.
Then the block matrix function 
$F(\cdot) = \left( \begin{smallmatrix} F_1(\cdot) & \mathbb{O}_{\ell \times m} \\ 
\mathbb{O}_{m \times \ell} & F_2(\cdot) \end{smallmatrix} \right)$
is convex and 
\begin{multline} \label{eq:BlockMatrixSubdiffIncl}
  \begin{pmatrix} 
    \partial F_1(\cdot) & \mathbb{O}_{\ell \times m} 
    \\
    \mathbb{O}_{m \times \ell} & \partial F_2(\cdot)
  \end{pmatrix}
  := \Big\{ V = (V^{(1)}, \ldots, V^{(d)}) \in (\mathbb{S}^{\ell + m})^d \Bigm|
  V^{(i)} = \left( \begin{smallmatrix} 
		     V_1^{(i)} & \mathbb{O}_{\ell \times m} \\ \mathbb{O}_{m \times \ell} & V_2^{(i)}
		   \end{smallmatrix} \right),  
  \\
  i \in \{ 1, \ldots, d \}, \enspace V_1 \in \partial F_1(\cdot), \: V_2 \in \partial F_2(\cdot) \Big\}
  \subseteq \partial F(\cdot).
\end{multline}
Moreover, if $V = (V^{(1)}, \ldots, V^{(d)}) \in \partial F(x)$ with 
$V^{(i)} = \left( \begin{smallmatrix} V^{(i)}_{11} & V^{(i)}_{12} \\ V^{(i)}_{21} & V^{(i)}_{22} \end{smallmatrix}
\right)$ for some $x \in \mathbb{R}^d$ and some matrices $V^{(i)}_{11} \in \mathbb{S}^{\ell}$, 
$V^{(i)}_{22} \in \mathbb{S}^m$, and $V^{(i)}_{12} = (V_{21}^{(i)})^T \in \mathbb{R}^{\ell \times m}$, then 
$V_{11} = (V^{(1)}_{11}, \ldots, V^{(d)}_{11}) \in \partial F_1(x)$ and 
$V_{22} = (V^{(1)}_{22}, \ldots, V^{(d)}_{22}) \in \partial F_2(x)$.
\end{theorem}

\begin{proof}
The convexity of the function $F$ can be easily verified directly. Let us prove the inclusion for the subdifferentials.

Fix any $x \in \mathbb{R}^d$ and $V_k \in \partial F_k(x)$. By definition for any $z_1 \in \mathbb{R}^{\ell}$ and 
$z_2 \in \mathbb{R}^m$ one has
\[
  \langle z_k, \Big( F_k(y) - F_k(x) + \sum_{i = 1}^d (y^{(i)} - x^{(i)}) V_k^{(i)} \Big) z_k \rangle \ge 0 
  \quad \forall y \in \mathbb{R}^d, \: k \in \{ 1, 2 \}.
\]
Summing up these two inequalities one obtains for any $z \in \mathbb{R}^{\ell + m}$ and $y \in \mathbb{R}^d$
\[
  \left\langle z, \left( F(y) - F(x) - \sum_{i = 1}^d (y^{(i)} - x^{(i)})
  \left( \begin{smallmatrix} 
  V_1^{(i)} & \mathbb{O}_{\ell \times m} \\ \mathbb{O}_{m \times \ell} & V_2^{(i)} 
  \end{smallmatrix} \right) 
   \right) z \right\rangle \ge 0,
\]
which implies that inclusion \eqref{eq:BlockMatrixSubdiffIncl} holds true.

Suppose now that $V = (V^{(1)}, \ldots, V^{(d)}) \in \partial F(x)$ with 
$V^{(i)} = \left( \begin{smallmatrix} V^{(i)}_{11} & V^{(i)}_{12} \\ V^{(i)}_{21} & V^{(i)}_{22} \end{smallmatrix}
\right)$ for some matrices $V^{(i)}_{11} \in \mathbb{S}^{\ell}$, $V^{(i)}_{22} \in \mathbb{S}^m$, and 
$V^{(i)}_{12} = (V_{21}^{(i)})^T \in \mathbb{R}^{\ell \times m}$. Then by
Lemma~\ref{lem:SubdiffQuadFormCharacterization} for any $z = (z_1, z_2) \in \mathbb{R}^{\ell} \times \mathbb{R}^m$ one
has
\begin{multline} \label{eq:BlockMatrixSubdiffScalarized}
  \Big( \langle z, V^{(1)} z \rangle, \ldots, \langle z, V^{(d)} z \rangle \Big)
  \\
  = \Big( \langle z_1, V_{11}^{(1)} z_1 \rangle + \langle z_1, V_{12}^{(1)} z_2 \rangle 
  + \langle z_2, V_{21}^{(1)} z_1 \rangle + \langle z_2, V_{22}^{(1)} z_2 \rangle, \ldots,
  \\
  \langle z_1, V_{11}^{(d)} z_1 \rangle + \langle z_1, V_{12}^{(d)} z_2 \rangle 
  + \langle z_2, V_{21}^{(d)} z_1 \rangle + \langle z_2, V_{22}^{(d)} z_2 \rangle \Big)
  \\
  \in \partial F_z(x) = \partial \Big( (F_1)_{z_1} + (F_2)_{z_2} \Big)(x)
  = \partial (F_1)_{z_1}(x) + \partial (F_2)_{z_2}(x).
\end{multline}
Putting $z_2 = 0$ one gets that 
\[
  \Big( \langle z_1, V_{11}^{(1)} z_1 \rangle, \ldots, \langle z_1, V_{11}^{(d)} z_1 \rangle \Big) 
  \in \partial (F_1)_{z_1}(x) \quad \forall z_1 \in \mathbb{R}^{\ell},
\]
which by Lemma~\ref{lem:SubdiffQuadFormCharacterization} implies that 
$V_{11} \in \partial F_1(x)$. In turn, putting $z_1 = 0$ and applying Lemma~\ref{lem:SubdiffQuadFormCharacterization}
once again one gets that $V_{22} \in \partial F_2(x)$.
\end{proof}

\begin{corollary} \label{crlr:DiagMatrixSubdiff}
Let $F(\cdot) = \diag(f_1(\cdot), \ldots, f_{\ell}(\cdot))$ for some real-valued convex functions 
$f_i \colon \mathbb{R}^d \to \mathbb{R}$. Then $F$ is convex and
\begin{multline*}
  \Big\{ (V^{(1)}, \ldots, V^{(d)}) \in (\mathbb{S}^{\ell})^d \Bigm| 
  V^{(i)} = \diag(v_1^{(i)}, \ldots, v_{\ell}^{(i)}), 
  \\
  v_j \in \partial f_j(\cdot), \: j \in \{ 1, \ldots, d \} \Big\}
  \subseteq \partial F(\cdot).
\end{multline*}
\end{corollary}

\begin{corollary}
If under the assumptions of Theorem~\ref{thrm:BlockMatrixSubdiff} one of the functions $F_1$ or $F_2$ is Fr\'{e}chet
differentiable at some $x \in \mathbb{R}^d$, then
\[
  \partial F(x) = 
  \begin{pmatrix} 
    \partial F_1(x) & \mathbb{O}_{\ell \times m} 
    \\
    \mathbb{O}_{m \times \ell} & \partial F_2(x)
  \end{pmatrix}.
\]
\end{corollary}

\begin{proof}
Suppose that, say, the function $F_2$ is Fr\'{e}chet differentiable at some point $x \in \mathbb{R}^d$. Let
$V = (V^{(1)}, \ldots, V^{(d)}) \in \partial F(x)$ with 
$V^{(i)} = \left( \begin{smallmatrix} V^{(i)}_{11} & V^{(i)}_{12} \\ 
V^{(i)}_{21} & V^{(i)}_{22} \end{smallmatrix} \right)$ 
for some matrices $V^{(i)}_{11} \in \mathbb{S}^{\ell}$, $V^{(i)}_{22} \in \mathbb{S}^m$, and 
$V^{(i)}_{12} = (V_{21}^{(i)})^T \in \mathbb{R}^{\ell \times m}$. Let us show that 
$V_{21}^{(i)} = \mathbb{O}_{m \times \ell}$ for all $i \in \{ 1, \ldots, d \}$. Then taking into account 
Theorem~\ref{thrm:BlockMatrixSubdiff} one obtains the required result.

Indeed, by Lemma~\ref{lem:SubdiffQuadFormCharacterization} for any $z = (z_1, z_2) \in \mathbb{R}^{\ell} \times
\mathbb{R}^m$ relations \eqref{eq:BlockMatrixSubdiffScalarized} hold true. Consequently, applying
Theorem~\ref{thrm:BlockMatrixSubdiff} and Proposition~\ref{prp:SubdiffAtDiffPoint} one obtains that 
\begin{multline*}
  Q(z_1, z_2) := \Big( \langle z_1, V_{11}^{(1)} z_1 \rangle + 2 \langle z_2, V_{21}^{(1)} z_1 \rangle, \ldots,
  \langle z_1, V_{11}^{(d)} z_1 \rangle + 2 \langle z_2, V_{21}^{(d)} z_1 \rangle \Big)
  \\
  \in \partial (F_1)_{z_1}(x)
\end{multline*}
for any $z = (z_1, z_2) \in \mathbb{R}^{\ell} \times \mathbb{R}^m$. 

Suppose by contradiction that $V_{21}^{(i)} \ne \mathbb{O}_{m \times \ell}$ for some $i \in \{ 1, \ldots, d \}$. 
Then there exists $z_1 \in \mathbb{R}^{\ell}$ such that $V_{21}^{(i)} z_1 \ne 0$. For the vector 
$z_2 = t V_{21}^{(i)} z_1$ one
has $Q(z_1, t V_{21}^{(i)} z_1) \in \partial (F_1)_{z_1}(x)$ for any $t \in \mathbb{R}$. However, note that
$\| Q(z_1, t V_{21}^{(i)} z_1) \| \to + \infty$ as $t \to \infty$, since $V_{21}^{(i)} z_1 \ne 0$ and
\[
  Q(z_1, t V_{21}^{(i)} z_1)^{(i)} = \langle z_1, V_{11}^{(i)} z_1 \rangle + 2 t \| V_{21}^{(i)} z_1 \|^2,
\]
which contradicts the fact that the subdifferential $\partial (F_1)_{z_1}(x)$ is a bounded set. Thus,
$V_{21}^{(i)} = \mathbb{O}_{m \times \ell}$ for all $i \in \{ 1, \ldots, d \}$.
\end{proof}

It should be noted that the inclusions for subdifferentials from Theorem~\ref{thrm:BlockMatrixSubdiff} and
Corollary~\ref{crlr:DiagMatrixSubdiff} need not be satisfied as equalities. In particular, a subgradient of a convex
diagonal matrix-valued function need not be a diagonal matrix.

\begin{example}
Let $d = 1$, $\ell = 2$, and 
$F(x) = \left( \begin{smallmatrix} \max\{ 0, 2x \} & 0 \\ 0 & \max\{ 0, 2x \} \end{smallmatrix} \right)$ 
for any $x \in \mathbb{R}$. Let us show that 
$V_0 = \left( \begin{smallmatrix} 1 & 1 \\ 1 & 1 \end{smallmatrix} \right) \in \partial F(0)$. 
To this end we will utilise Theorem~\ref{thrm:OneDimCase}. Indeed, one has
\[
  F'_+(0) = \begin{pmatrix} 2 & 0 \\ 0 & 2 \end{pmatrix}, \quad
  F'_-(0) = \mathbb{O}_{2 \times 2}.
\]
It is easily seen that the eigenvalues of the matrices
\[
  F'_+(0) - V_0 = \begin{pmatrix} 1 & -1 \\ -1 & 1 \end{pmatrix}, \quad
  V_0 - F'_-(0) = \begin{pmatrix} 1 & 1 \\ 1 & 1 \end{pmatrix}
\]
are $0$ and $2$. Therefore, these matrices are positive semidefinite or, equivalently, 
$F'_+(0) \succeq V_0 \succeq F'_-(0)$, which by Theorem~\ref{thrm:OneDimCase} implies that $V_0 \in \partial F(0)$. 
Note that the matrix $V_0$ is not diagonal, despite the fact that the function $F$ is diagonal matrix-valued.

Observe also that
\[
  \partial_{Cl} F(0) 
  = \left\{ \begin{pmatrix} t & 0 \\ 0 & t \end{pmatrix} \in \mathbb{S}^2 \Biggm| t \in [0, 2] \right\}
\]
and, therefore, $\partial_{Cl} F(0) \ne \partial F(0)$, that is, the inclusion from
Theorem~\ref{thrm:ConvAnalSubdiff_vs_ClarkSubdiff}
is strict in the general case.

In addition, note that
\[
  F = F_1 + F_2, \quad 
  F_1(x) = \left( \begin{smallmatrix} \max\{ 0, 2x \} & 0 \\ 0 & 0 \end{smallmatrix} \right), \quad
  F_2(x) = \left( \begin{smallmatrix} 0 & 0 \\ 0 & \max\{ 0, 2x \} \end{smallmatrix} \right).
\]
Let us show that $V_0 \notin \partial F_1(0) + \partial F_2(0)$, that is, 
$\partial F(0) \ne \partial F_1(0) + \partial F_2(0)$, which means that the inclusion from
Proposition~\ref{prp:SumSubdiff} is also strict in the general case.

With the use of Theorem~\ref{thrm:OneDimCase} one gets
\begin{equation}
\begin{split}
  &\partial F_1(0) = \Big\{ V \in \mathbb{S}^2 \Bigm| (F_1)'_-(0) \preceq V \preceq (F_1)'_+(0) \Big\},
  \\ 
  &(F_1)'_-(0) = \mathbb{O}_{2 \times 2}, \quad 
  (F_1)'_+(0) = \left( \begin{smallmatrix} 2 & 0 \\ 0 & 0 \end{smallmatrix} \right).
\end{split}
\end{equation}
Let us show that for any $V = \{ V_{ij} \}_{i, j = 1}^2 \in \partial F_1(0)$ one has $V_{12} = V_{22} = 0$. Indeed,
applying inequality $(F_1)'_-(0) \preceq V$ and relations \eqref{eq:LownerOrderViaQuadForm} with $z = (0, 1)$ one gets
$0 \le \langle z, (V - (F_1)'_-(0)) z \rangle = V_{22}$, while inequality $V \preceq (F_1)'_+(0)$ implies that
$0 \le \langle z, ((F_1)'_+(0) - V) z \rangle = -V_{22}$. Thus, $V_{22} = 0$. Note also that the inequality 
$V \succeq (F_1)'_-(0)$ means that the matrix $V$ is positive semidefinite. Therefore, its determinant 
$\determ(V) = - V_{12}^2$ is nonnegative, which implies that $V_{12} = 0$.

Arguing in the same way one can show that for any $V \in \partial F_2(0)$ one has $V_{11} = V_{12} = 0$. Therefore,
matrices from the set $\partial F_1(0) + \partial F_2(0)$ are diagonal, which implies that the non-diagonal matrix 
$V_0 \in \partial F(0)$ does not belong to this set.
\end{example}

Theorem~\ref{thrm:BlockMatrixSubdiff} admits extensions to various classes of convex block matrix functions,
whose blocks are constructed from convex matrix-valued functions. Let us present one such extension.

\begin{theorem}
The block matrix function 
$G(\cdot) = \left( \begin{smallmatrix} F(\cdot) & - F(\cdot) \\ - F(\cdot) & F(\cdot) \end{smallmatrix} \right)$
is convex and 
\begin{multline} \label{eq:BlockFromOneMatrixSubdiffIncl}
  \Big\{ W = (W^{(1)}, \ldots, W^{(d)}) \in (\mathbb{S}^{2\ell})^d \Bigm|
  W^{(i)} = \left( \begin{smallmatrix} 
		     V^{(i)} & -V^{(i)} \\ -V^{(i)} & V^{(i)}
		   \end{smallmatrix} \right), 
  \\
  i \in \{ 1, \ldots, d \}, \enspace V \in \partial F(\cdot) \Big\} \subseteq \partial G(\cdot).
\end{multline}
Moreover, if $V = (V^{(1)}, \ldots, V^{(d)}) \in \partial G(x)$ with 
$V^{(i)} = \left( \begin{smallmatrix} V^{(i)}_{11} & V^{(i)}_{12} \\ V^{(i)}_{21} & V^{(i)}_{22} \end{smallmatrix}
\right)$ for some $x \in \mathbb{R}^d$ and some matrices $V^{(i)}_{11}, V^{(i)}_{22} \in \mathbb{S}^{\ell}$ and 
$V^{(i)}_{12} = (V_{21}^{(i)})^T \in \mathbb{R}^{\ell \times \ell}$, then 
$V_{11} = (V^{(1)}_{11}, \ldots, V^{(d)}_{11}) \in \partial F(x)$ and 
$V_{22} = (V^{(1)}_{22}, \ldots, V^{(d)}_{22}) \in \partial F(x)$.
\end{theorem}

\begin{proof}
For any $z = (z_1, z_2) \in \mathbb{R}^{\ell} \times \mathbb{R}^{\ell}$ one has
\begin{align*}
  G_z(\cdot) = \langle z, G(\cdot) z \rangle 
  &= \langle z_1, F(\cdot) z_1) \rangle - 2 \langle z_1, F(\cdot) z_2 \rangle + \langle z_2, F(\cdot) z_2 \rangle
  \\
  &= \langle (z_1 - z_2), F(\cdot) (z_1 - z_2) \rangle = F_{z_1 - z_2}(\cdot).
\end{align*}
Since the matrix-valued function $F$ is convex, the function $F_{z_1 - z_2}(\cdot)$ is convex for any 
$z_1, z_2 \in \mathbb{R}^d$. Therefore, the function $G_z(\cdot)$ is convex for any $z \in \mathbb{R}^{2 \ell}$, 
which implies that the block matrix function $G$ is also convex.

Choose any $x \in \mathbb{R}^d$ and $V \in \partial F(x)$. By the definition of subdifferential and relations 
\eqref{eq:LownerOrderViaQuadForm} one has
\begin{align*}
  0 &\le \Big\langle (z_1 - z_2), \Big( F(y) - F(x) 
  - \sum_{i = 1}^d (y^{(i)} - x^{(i)}) V^{(i)} \Big) (z_1 - z_2) \Big\rangle
  \\
  &= \left\langle z, \left( G(y) - G(x) - \sum_{i = 1}^d (y^{(i)} - x^{(i)}) 
  \left( \begin{smallmatrix} V^{(i)} & - V^{(i)} 
  \\ - V^{(i)} & V^{(i)} \end{smallmatrix} \right) \right) z \right\rangle
\end{align*}
for any $y \in \mathbb{R}^d$ and $z = (z_1, z_2) \in \mathbb{R}^{\ell} \times \mathbb{R}^{\ell}$. Consequenlty,
inclusion \eqref{eq:BlockFromOneMatrixSubdiffIncl} is valid.

Let now $V \in \partial G(x)$ be from the formulation of the second statement of the theorem. Then by the definition
of subdifferential one has
\[
  \Big\langle z, \Big( G(y) - G(x) - \sum_{i = 1}^d (y^{(i)} - x^{(i)}) V^{(i)} \Big) z \Big\rangle \ge 0
\]
for all $y \in \mathbb{R}^d$ and $z = (z_1, z_2) \in \mathbb{R}^{\ell} \times \mathbb{R}^{\ell}$. Putting $z_2 = 0$
one gets
\[
  \Big\langle z_1, \Big( F(y) - F(x) - \sum_{i = 1}^d (y^{(i)} - x^{(i)}) V_{11}^{(i)} \Big) z_1 \Big\rangle \ge 0
\]
for all $y \in \mathbb{R}^d$ and $z_1 \in \mathbb{R}^{\ell}$. Therefore $V_{11} \in \partial F(x)$. The inclusion
$V_{22} \in \partial F(x)$ is proved in the same way.
\end{proof}

Let us finally return to Example~\ref{ex:ConvexFunctWithNonconvexElements}, from which we started our discussion of
subdifferentials of convex matrix-valued functions.

\begin{example}
Let $d = \ell = 2$ and 
\[
  F(x) = \begin{pmatrix} 
	  |x^{(1)}| + |x^{(2)}| & |x^{(1)}| - |x^{(2)}| 
	  \\ 
	  |x^{(1)}| - |x^{(2)}| & |x^{(1)}| + |x^{(2)}|  
	 \end{pmatrix}
  \quad \forall x = (x^{(1)}, x^{(2)}) \in \mathbb{R}^2
\]
Let us compute a subset of the subdifferential of this convex matrix-valued mapping at the origin. A subset of the
subdifferential of $F$ at any other point can be computed in precisely the same way.

For the sake of convenience, introduce the functions
\[
  F_1(x) = \begin{pmatrix} 
	    |x^{(1)}|& |x^{(1)}| 
	    \\ 
	    |x^{(1)}| & |x^{(1)}|
  \end{pmatrix}
  \quad
  F_2(x) = \begin{pmatrix} 
	    |x^{(2)}| & - |x^{(2)}| 
	    \\ 
	    - |x^{(2)}| & |x^{(2)}|  
  \end{pmatrix} 
  \quad \forall x = (x^{(1)}, x^{(2)}) \in \mathbb{R}^2  
\]
It is easily seen that both $F_1$ and $F_2$ are convex. Since $F = F_1 + F_2$, by Proposition~\ref{prp:SumSubdiff} and
Theorem~\ref{thrm:ConvAnalSubdiff_vs_ClarkSubdiff} one has 
$\partial_{Cl} F_1(0) + \partial_{Cl} F_2(0) \subseteq \partial F(0)$. One can readily check that
\[
  \partial_{Cl} F_1(0) + \partial_{Cl} F_2(0) = \left\{
    \left( \left( \begin{smallmatrix} t & t \\ t & t \end{smallmatrix} \right), 
    \left( \begin{smallmatrix} s & -s \\ -s & s \end{smallmatrix} \right) \right) \in \mathbb{S}^2 \times \mathbb{S}^2
    \Bigm| t, s \in [-1, 1]
  \right\}
\]
Any pair of matrices from this set is a subgradient (in the sense of convex analysis) of $F$ at zero.
\end{example}

\bibliographystyle{abbrv}  
\bibliography{MatrixSubdifferential_bibl}

\end{document}